\documentclass[a4paper,oneside,11pt]{article}
\usepackage[english]{babel} %francais, polish, spanish, ...
\usepackage[T1]{fontenc}
\usepackage[utf8]{inputenc}
\usepackage{lmodern} %Type1-font for non-english texts and characters
\usepackage[numbers]{natbib}
\usepackage{hyphenat}
\usepackage{amsmath}
\usepackage{amsthm}
\usepackage{amsfonts}
\newtheorem{teo}{Theorem}[section]
\newtheorem{prop}[teo]{Proposition}
\newtheorem{lemma}[teo]{Lemma}
\newtheorem{re}[teo]{Remark}

\newtheorem{cor}[teo]{Corollary}
\numberwithin{equation}{section}
\newcommand{\Real}{\mathbb R}

\newcommand{\Img}{\mbox{ Im}\,}

\usepackage{mathtools}
\mathtoolsset{showonlyrefs}

\usepackage[plainpages=false,pdfpagelabels,pdfusetitle,pdfdisplaydoctitle, pdfstartpage=1, pdfstartview={{XYZ null null 1.00}}]{hyperref}
\usepackage[top=3cm, bottom=2cm, left=3cm, right=2cm]{geometry}

\author{\textsc{Luccas Campos and Mykael Cardoso}}

\begin{document}

\title{On the critical norm concentration for the inhomogeneous nonlinear Schrödinger equation}
\date{}

\maketitle
\begin{abstract}\noindent
We consider the inhomogeneous nonlinear Schrödiger equation (INLS) in $\mathbb{R}^N$
$$i \partial u_t + \Delta u + |x|^{-b} |u|^{2\sigma}u = 0,$$
%with $0 < b < \frac{N}{3}$, if $N \leq 3$ and $0 < b < 2$, if $N \geq 4$.
and show the $L^2$-norm concentration for the finite time blow-up solutions in the $L^2$-critical case, $\sigma=\frac{2-b}{N}$. Moreover, we provide an alternative for the classification of minimal mass blow-up solutions first proved by Genoud and Combet \citep{Cmb}. For the case $\frac{2-b}{N} < \sigma < \frac{2-b}{N-2}$, we show results regarding the $L^p$-critical norm concentration, generalizing the argument of Holmer and Roudenko \citep{holmer2007blow} to the INLS setting.
\end{abstract}

\section{Introduction}\label{s_Introduction}
In this paper, we consider the initial value problem (IVP) for the inhomogenous nonlinear Schr\"odinger (INLS) equation 
\begin{equation}
\begin{cases}
i \partial u_t + \Delta u + |x|^{-b} |u|^{2 \sigma}u = 0, \,\,\, x \in \mathbb{R}^N, \,t>0,\\
u(\cdot,0) = u_0 \in H^1(\mathbb{R}^N),
\end{cases}
\label{INLS}
\end{equation}
where $\sigma>0$ and $b>0$. This model arises naturally in nonlinear optics for the propagation of laser beams.
The case $b = 0$ is the classical nonlinear Schrödinger equation (NLS), extensively studied in recent years (see Sulem-Sulem \citep{Sulem}, Bourgain \citep{Bo99}, Cazenave \citep{cazenave}, Linares-Ponce \citep{LiPo15}, Fibich \citep{Fi15} and the references therein).
The critical Sobolev index where one can expect well-posedness for this model is given by scaling. First, note that if
$u(x, t)$ is a solution of (\ref{INLS}), then
\begin{equation}\label{Scalinv}
u_\rho(x,t) = \rho^\frac{2-b}{2\sigma} u(\rho x, \rho^2 t)
\end{equation}
is also a solution, for all $\rho > 0$. Computing the homogeneous Sobolev norm, we have
\begin{equation*}
\|u_\rho(\cdot,0)\|_{\dot{H}^s} = \rho^{s - \left(\frac{N}{2} - \frac{2-b}{2\sigma}\right)}\|u_0\|_{\dot{H}^s}.
\end{equation*}
Hence, the critical Sobolev index is the one which leaves the scaling symmetry invariant, that is
$$
s_c = \frac{N}{2}-\frac{2-b}{2\sigma}.
$$ 
In this paper, we are interested in the case $0 \leq s_c <1$. The case $s_c = 0$ is known as the \textit{$L^2$-critical}. Rewriting this condition in terms of $\sigma$, we obtain
\begin{equation*}
\sigma = \frac{2-b}{N}.
\end{equation*}
On the other hand, the case $0<s_c <1$ is known as the \textit{$L^2$-supercritical} and \textit{$H^1$-subcritical} (or just \textit{intercritical}). Again, we can reformulate this condition in terms of $\sigma$ as
$$
\frac{2-b}{N}<\sigma<\sigma^*_b,
$$
where
\begin{equation}\label{sigma_*}
\sigma^*_b = \begin{cases} \infty, & N \leq 2  \\
\frac{2-b}{N-2}, & N \geq 3. \\
\end{cases}
\end{equation}
The local well-posedness for the INLS equation was first studied by Genoud-Stuart in \citep{g_8} (see also Genoud \citep{g_6}) by the abstract theory of Cazenave \citep{cazenave}, without relying on Strichartz type inequalities. They analyzed the IVP \eqref{INLS} in the sense of distributions, that is, $i\partial_tu+\Delta u+|x|^{-b}|u|^{2\sigma}u=0 $ in $H^{-1}(\mathbb{R}^N)$ and showed, with $0 < b < 2$, it is well-posed
\begin{itemize}
        \item[-] locally if $0 < \sigma < \sigma^*_b$ ($s_c < 1$);
        \item[-] globally for any initial data in $H^1(\Real^{N})$ if $\sigma < \frac{2-b}{N}$ ($s_c < 0$);
        \item[-] globally for sufficiently small initial data if $\frac{2-b}{N}  \leq \sigma < \sigma_b^*$ ($0 \leq s_c < 1$).
\end{itemize}
More recently, Guzmán \citep{Boa} established local well-posedness of the INLS in $H^s(\Real^{N})$ based on Strichartz estimates. In particular, setting 
$$
\tilde{2} = \begin{cases} \frac{N}{3}, & N \leq 3  \\
2, & N \geq 4, \\
\end{cases}
$$
he proved that, for $N \geq 2$, $0 < \sigma < \sigma_b^*$ and $0 < b < \tilde{2}$, the initial value problem \eqref{INLS} is locally well-posed in $H^1(\Real^N)$. Dinh \citep{Boa_Dinh} improved Guzmán's results in dimensions $N=2$ (for $0<b<1$ and $0 < \sigma < \sigma_b^*)$ and  $N = 3$ (for $0 < b < \frac{3}{2}$ and $0 < \sigma < \frac{3-2b}{2b-1}$). Note that the results of Guzmán \citep{Boa} and Dinh \citep{Boa_Dinh} do not treat the case $N = 1$, and the ranges of $b$ are more restricted than those in the results of Genoud-Stuart \citep{g_8}. However, Guzmán and Dinh give more detail information on the solutions, showing that there exists $T(\|u_0\|_{H^1})>0$ such that $u \in L^q\left([-T,T];L^r(\Real^N)\right)$ for any $L^2$-admissible pair $(q,r)$ satisfying
\begin{equation*}
\frac{2}{q}=\frac{N}{2}-\frac{N}{r},
\end{equation*}
where 
\begin{equation*}
\left\{\begin{array}{cl}
2\leq & r  \leq \frac{2N}{N-2}\hspace{0.5cm}\textnormal{if}\;\;\;  N\geq 3,\\
2 \leq  & r < +\infty\;  \hspace{0.5cm}\textnormal{if}\;\; \;N=2,\\
2 \leq & r \leq + \infty\;  \hspace{0.5cm}\textnormal{if}\;\;\;N=1.
\end{array}\right.
\end{equation*}

Let $T^+(u)$ be the maximal positive time of existence for a solution $u$ to \eqref{INLS}. To simplify the notation we only write $T$ to denote $T^+(u)$. If $T=+\infty$, we say that the solution is global, and if $T<+\infty$, we say that the solution $u$ blows up in finite time. In the latter case, a scaling argument gives the following lower bound
\begin{equation}\label{lowerbound}
\|\nabla u(t)\|_{L^2} \geq \frac{C}{(T - t)^\frac{1-s_c}{2}     }.
\end{equation}
In particular, $\|u(t)\|_{H^1}\to +\infty$ when $t\nearrow T$.

The solutions to (\ref{INLS}) have the following conserved quantities
\begin{equation}\label{mass}
M\left[u(t) \right] = \int |u(t)|^2 dx = M[u_0],
\end{equation}
\begin{equation}\label{energy}
E\left[u(t) \right] = \frac{1}{2}\int |\nabla u(t)|^2 dx - \frac{1}{2 \sigma+2} \int |x|^{-b}|u(t)|^{2 \sigma+2} dx = E[u_0].
\end{equation}

In addition to the \emph{scaling invariance} given by \eqref{Scalinv},  there are more symmetries for \eqref{INLS}. Indeed, if $u(x,t)$ is a solution of \eqref{INLS}, so are the following
\begin{itemize}
\item[(a)] $v_1(x,t)=u(x,t-t_0)$, for all $t_0\in \Real$ (\emph{time translation invariance}),
\item[(b)] $v_2(x,t)=e^{i\gamma}u(x,t)$ for all $\gamma\in \Real$ (\emph{phase invariance}).
%and\item[c)] $v(x,t)= \lambda^{\frac{2-b}{2\sigma}}u( \lambda x,\lambda^2t)$ (\emph{scaling invariance}).
\end{itemize}
Moreover, in the $L^2$-critical case $\left(\sigma=\frac{2-b}{N}\right)$, we have one more invariance, the so-called pseudoconformal transformation (see Combet-Genoud \cite[Lemma 9]{Cmb})
\begin{equation*}
v(x,t)=\frac{1}{|t|^{\frac{N}{2}}}\overline{u}\left(\frac{x}{t},\frac{1}{t}\right)e^{-i\frac{|x|^2}{4t}}.
\end{equation*}

Note that, unlike the NLS equation, the space translation invariance is broken for the INLS model due to the presence of the term $|x|^{-b}$ in the nonlinear term.

Also, if $u_0 \in \Sigma = \left\{ f \in H^1(\mathbb{R}^N) ; |x|f \in L^2(\mathbb{R}^N)\right\}$, then the solution satisfies the \textit{virial identity} (see Farah \citep[Proposition 4.1]{Farah_well})
\begin{equation}
\frac{d^2}{dt^2} \int |x|^2 |u|^2\, dx = 8 (N \sigma +b)E\left[u_0 \right]-4(N \sigma +b -2) \int \left|\nabla u \right|^2\,dx.
\label{virial}
\end{equation}

From this identity, we immediately see that, if $\sigma > \frac{2-b}{N}$ and $E\left[u_0 \right] < 0$, then the graph of $t \mapsto \int |x|^2 |u|^2$ lies below an inverted parabola, which becomes negative in finite time. Therefore, the solution cannot exist globally and blows up in finite time.

The blow-up theory is related to the concept of \textit{ground state}, which is the unique positive radial solution of the elliptic problem
\begin{equation}
\Delta Q - Q + |x|^{-b}|Q|^{2 \sigma} Q = 0.
\label{ground_state}
\end{equation}

The existence of the ground state is proved by Genoud-Stuart \citep{g_5, g_8} for dimension $N \geq 2$, and by Genoud \citep{g_6} for $N = 1$. Uniqueness was proved in dimension $N \geq 3$ by Yanagida \citep{g_19} (see also Genoud \citep{g_5}), in dimension $N = 2$ by Genoud \citep{g_7} and in dimension $N = 1$ by Toland \citep{g_16}. The existence and uniqueness hold for $0 < b < \tilde{2}$ and $0 < \sigma < \sigma^*_b$.

The ground state satisfies the following Pohozaev's identities (see relations (1.9)-(1.10) in Farah \citep{Farah_well})
\begin{align}
&\int |\nabla Q|^{2}\,dx= \left(\frac{N\sigma + b}{2\sigma +2 -(N\sigma+b)}\right)\|Q\|_{L^2}^2,
\label{pohozaev}\\
&\int |x|^{-b}|Q|^{2\sigma+2}\,dx = \left(\frac{2\sigma + 2}{2\sigma +2 -(N\sigma+b)}\right) \|Q\|_{L^2}^2.
\label{pohozaev2}
\end{align}

In \cite{Farah_well}, Farah proved the following sharp Gagliardo-Nirenberg inequality, valid for $0 \leq s_c < 1$ and $0 < b < 2$
\begin{equation}
\int |x|^{-b}|u|^{2\sigma+2}\,dx \leq K_{opt} \|\nabla u\|_{L^2}^{N\sigma +b} \| u\|_{L^2}^{2\sigma + 2-(N\sigma+b)},
\label{gagliardo}
\end{equation}

where the sharp constant $K_{opt}$ is given by
\begin{equation}
K_{opt} = \left(\frac{N\sigma + b}{2\sigma +2 -(N\sigma+b)}\right)^\frac{2-(N\sigma+b)}{2} \frac{2\sigma +2}{(N\sigma + b) \|Q\|_{L^2}^{2\sigma}}.
\label{gagliardo_opt}
\end{equation}

This inequality can be seen as an extension to the case $b > 0$ of the classical Gagliardo-Nirenberg inequality. It is also an extension of the inequality obtained by Genoud \citep{Genoud1}, who showed its validity for $\sigma = \frac{2-b}{N}$.

We will first consider the $L^2$-critical case. By  energy conservation \eqref{energy} and Gagliardo-Nirenberg inequality \eqref{gagliardo}, if $u(t)$ is the $H^1(\Real^N)$-solution of \eqref{INLS} with initial data satisfying $\|u_0\|_{L^2}< \|Q\|_{L^2}$, then
$$E[u_0]\geq\frac{1}{2}\|\nabla u(t)\|_{L^2}^2\left[1-\left(\frac{\left\|u_0\right\|_{L^2}}{\|Q\|_{L^2}}\right)^{\frac{4-2b}{N}}\right],$$
which implies that $u(t)$ is global, i.e., $T=+\infty$. Thus, the only possible finite time blow-up solutions of \eqref{INLS} must have $\|u_0\|_{L^2}\geq \|Q\|_{L^2}$. Moreover, we say that $u$ is a mass critical solution if $\|u_{0}\|_{L^2}=\|Q\|_{L^2}$. Merle-Tsutsumi \citep{MT} (in the radial case) and Weinstein \citep{W_1989} (in the general case), considered the $L^2$-critical NLS and showed that every finite time blow-up solution must concentrate the $L^2(\mathbb{R}^N)$ norm.

In this paper, in the same spirit of Hmidi-Keraani \citep{keraani}, we also prove mass concentration for the $L^2$-critical INLS equation. More precisely, we show the following.
\begin{teo}\label{massconcentration}
Let $u$ be a solution of (\ref{INLS}) in the $L^2$-critical case ($\sigma=\frac{2-b}{N}$), which blows up in finite positive time $T$ and $\lambda(t) > 0$ be any function such that $\lambda(t) \|\nabla u(t)\|_2 \rightarrow +\infty$ as $t\nearrow T$. Then, there exists $x(t) \in \mathbb{R}^N$ such that
\begin{equation*}
\liminf_{t \nearrow T} \int_{|x-x(t)|<\lambda(t)} |u(x,t)|^2 dx \geq \|Q\|^2_{L^2\left(\mathbb{R}^N\right)}.
\end{equation*}
\end{teo}
\begin{re} From lower bound on the blow-up rate \eqref{lowerbound}, it is possible to deduce the existence of $\lambda(t)$ such that $\lambda(t)\rightarrow 0$, as $t\rightarrow T$. Indeed, $\lambda(t)=(T-t)^{\alpha}$ with $0<\alpha<1/2$ satisfies the assumptions of Theorem \ref{massconcentration}.
\end{re}

The main ingredient to prove Theorem \ref{massconcentration} is the following compactness theorem.
\begin{teo}\label{teo1}
Let $\left\{v_n\right\}_{n = 1}^\infty$ be a bounded family in $H^1 \left(\mathbb{R}^N\right)$ such that
\begin{equation}
\limsup_{n \rightarrow  \infty} \|\nabla v_n\|_{L^2} \leq M, \,\,\,\,\, \limsup_{n \rightarrow \infty} \left\||\cdot|^\frac{-b}{2\sigma+2}v_n\right\|_{L^{2\sigma+2}} \geq m,
\label{teo1_bound}
\end{equation}
for some $m$, $M > 0$. Then there exist $\left\{x_n\right\}_{n = 1}^\infty \subset \mathbb{R}^N$ and $V \in H^1\left(\mathbb{R}^N\right)$ such that, up to a subsequence,
\begin{equation*}
v_n( \cdot + x_n ) \rightharpoonup V
\label{}
\end{equation*}

weakly in $H^1$. Moreover, in the mass-critical setting ($s_c = 0$), one has
\begin{equation}\label{eq_117}
\|V\|_{L^2} \geq C(M,m) \|Q\|_{L^2},
\end{equation}

where
\begin{equation}\label{CMn}
        C(M,m) = \left(\frac{m^{\frac{4-2b}{N}+2}N}{M^2(2-b+N)}\right)^\frac{N}{2(2-b)}.
\end{equation}
\end{teo}
\begin{re}
The lower bound on the $L^2(\mathbb{R}^N)$ norm of $V$ \eqref{eq_117} is sharp. Indeed, taking $v_n = Q$ one has, by \eqref{pohozaev} and \eqref{pohozaev2}, $m^{\frac{4-2b}{N}+2} = \left(\frac{N}{2-b}+1	\right) \|Q\|_{L^2}^2$ and $M^2 = \left(\frac{N}{2-b}\right)\|Q\|_{L^2}^2$. Therefore, $C(M,m)=1$, and \eqref{eq_117} is an equality.
\end{re}

\begin{re}
In the radial setting it is possible to give a simple proof of Theorem \ref{massconcentration} without relying in the compactness Theorem \ref{teo1}. See the details in Remark \ref{Rem1} below.
\end{re}
Using the pseudoconformal transformation applied to \emph{standing wave} $e^{it} Q$, and the three symmetries of \eqref{INLS} described previously, we obtain a three-parameter family $$\left(S_{T,\lambda,\gamma}\right)_{T\in \Real,\lambda>0,\gamma\in \Real}$$ of mass critical solutions for \eqref{INLS}, $\|u_0\|_{L^2}=\|Q\|_{L^2}$, blowing up in finite time, defined as
\begin{equation*}
S_{T,\lambda,\gamma}(x,t)=e^{i\gamma}e^{i\frac{\lambda^2}{T-t}}e^{-i\frac{|x|^2}{4(T-t)}}\left(\frac{\lambda_0}{T-t}\right)^{\frac{N}{2}}Q\left(\frac{\lambda x}{T-t}\right).
\end{equation*}

(see Combet-Genoud \citep[Proposition 10]{Cmb}).

In fact, these are the only finite-time blow up solutions with critical mass, as shown in the following theorem.
\begin{teo}\label{caract} Let $u$ be a solution of \eqref{INLS} in the $L^2$-critical case ($\sigma=\frac{2-b}{N}$) blowing up in finite time $T>0$, such that $\|u_0\|_{L^2}=\|Q\|_{L^2}$. Then, $e^{i(|x|^2/4T)}u_0\in \mathcal{A}$, where
\begin{equation*}
\mathcal{A}=\left\{\rho^{N/2}e^{i\theta}Q(\rho x),\,y\in \Real^N,\,\rho\in \Real^+_*,\,\theta\in [0,2\pi) \right\}.
\end{equation*}
\end{teo}

This theorem was first obtained by Combet-Genoud \citep[Theorem 1]{Cmb}. Here we give an alternative proof based in the ideas introduced by Hmidi-Keraani \citep{keraani} in the NLS setting.

Next, we consider the intercritical INLS equation. In \citep{holmer2007blow}, Holmer and Roudenko investigated the asymptotic behavior of the $L^{3}(\mathbb{R}^N)$ norm for radial, finite-time blowup solutions to the NLS equation in the 3D cubic case. We generalize this result for the INLS equation with $0 < s_c < 1$. We  show the concentration of the $L^{\sigma_c}(\mathbb{R}^N)$ norm, where
\begin{equation*}%\label{p_c}
        \sigma_c = \frac{2N\sigma}{2-b}
\end{equation*}

is such that $\dot{H}^{s_c}(\mathbb{R}^N) \hookrightarrow L^{\sigma_c}(\mathbb{R}^N)$. Note that, in the $L^2$-critical case, $\dot{H}^{s_c}(\mathbb{R}^N) = L^{\sigma_c} (\mathbb{R}^N)= L^2(\mathbb{R}^N)$.
\begin{teo}\label{Tlpc}
Let $N \geq 2$ and $\frac{2-b}{N} < \sigma < \min\{\frac{2-b}{N-2},2\}$. Suppose that $u$ is a radial $H^1(\Real^N)$-solution of \eqref{INLS} that blows up in finite time $T>0$. Then either there exist $c_0\gg 1 $ (depending on the solution, but not on time) such that for $t\to T$
\begin{equation}\label{fint}
\int_{|x|\leq c_0^2\|\nabla u(t)\|_{L^2}^{-\frac{1}{1-s_c}}}|u(x,t)|^{\sigma_c}\,dx\geq c_0^{-1},
\end{equation}
or there exist a sequence  of times $t_n\to T$ such that, for some constant $\widetilde{c}_0$ (independent on the solution and on time)
\begin{equation}\label{inft}
 \int_{|x|\leq \widetilde{c}_0\|u_0\|_{L^2}^{\frac{2(\sigma+2)}{2(\sigma(N-1)+b)}}\|\nabla u(t_n)\|_{L^2}^{-\frac{2(2-\sigma)}{2(\sigma(N-1)+b)}}}|u(x,t_n)|^{\sigma_c}\,dx\to +\infty.
\end{equation}
\end{teo}
\begin{re}
We can prove the above Theorem for the intercritical NLS and also generalizing the result of Holmer-Roudenko \citep{holmer2007blow} in the NLS setting. This will appear in a future work.
\end{re}

The two cases in the statement of Theorem \ref{Tlpc} are not mutually exclusive. The restriction on the dimension comes from using the decay of radial $H^1$ functions. From the lower bound \eqref{lowerbound}, the concentration window in \eqref{fint} satisfies 
\begin{equation*}
\| \nabla u(t)\|_{L^2}^{-\frac{1}{1-s_c}} \leq c (T-t)^\frac{1}{2},
\end{equation*}

and the concentration window \eqref{inft} satisfies 
\begin{equation*}
\|\nabla u(t)\|_{L^2}^{-\frac{2(2-\sigma)}{2(\sigma(N-1)+b)}} \leq c (T-t)^{\frac{(2-\sigma)(1-s_c)}{2(\sigma(N-1)+b)}}.
\end{equation*}

This paper is structured as follows. In sections \ref{s_Mass_Concentration} and \ref{s_Minimal_Mass_Solutions}, we prove Theorem \ref{massconcentration} and Theorem \ref{caract}, respectively, both assuming Theorem \ref{teo1}. Section \ref{s_Critical_Norm_Concentration} is devoted to the proof of Theorem \ref{Tlpc}. In the Appendix, we give a proof for Theorem \ref{teo1}.

\textbf{Acknowledgments}
The authors thank
Svetlana Roudenko (FIU) and Luiz Gustavo Farah (UFMG) for their valuable for valuable comments and suggestions which helped improve the manuscript.
Part of this work was done when the first author was visiting Florida International University in 2018-19 under the support of Coordenação de Aperfeiçoamento de Pessoal de Nível Superior - Brasil (CAPES), for which all authors are very grateful as it boosted the energy into
the research project.
L. C. was financed in part by the Coordenação de Aperfeiçoamento de Pessoal de Nível Superior - Brasil (CAPES) - Finance Code 001.

%Genoud (\citep{Genoud1}) showed the threshold for global existence. He also exhibited a class of solutions with minimal mass which blow-up in finite time, and so the threshold (\ref{l2_threshold}) is sharp. This result is, in fact, an extension for the INLS model of the classical global well-posedness proved by Weinstein (\citep{}) for the Nonlinear Schrödinger Equation (NLS), which is essentially (\ref{INLS}) with $b = 0$.

%Using this sharp Gagliardo Nirenberg, the threshold in the $L^2$ supercritical ($s_c > 0$) and $H^1$ subcritical ($s_c < 1$) was stablished, in the

\section{Mass Concentration}\label{s_Mass_Concentration}

In this section, we prove Theorem \ref{massconcentration}, assuming Theorem \ref{teo1}, which is proved in the  Appendix.
\begin{proof}[Proof of Theorem \ref{massconcentration}]
First, we define
\begin{equation*}
\rho(t) = \frac{\|\nabla Q\|_{L^2}}{\|\nabla u(t)\|_{L^2}}
\end{equation*}
and since $\sigma=\frac{2-b}{N}$
\begin{equation}\label{defv}
v(x,t) = \rho(t)^\frac{2-b}{2\sigma} u(\rho(t)x,t) = \rho(t)^\frac{N}{2} u(\rho(t)x,t).
\end{equation}

Let $\left\{t_n\right\}$ be an arbitrary sequence of times such that $t_n \nearrow T$ and define $\rho_n := \rho(t_n)$ and $v_n := v(\cdot, t_n)$. Note that $\rho_n\rightarrow 0$, as $n\rightarrow \infty$ by \eqref{lowerbound}. In view of the mass conservation \eqref{mass} and the definition of $\rho(t)$, the sequence $\left\{v_n\right\}$ satisfies
\begin{equation}\label{defv2}
        \|v_n\|_{L^2} = \|u_0\|_{L^2}, \,\,\, \|\nabla v_n \|_{L^2} = \|\nabla Q\|_{L^2}.
\end{equation}

Also, for the energy
\begin{equation*}
        E[v_n] = \rho_n^2 \left(\frac{1}{2}\|\nabla u(t_{_n})\|_{L^2}^2 -  \frac{N}{2(2-b + N)}\left\||\cdot|^\frac{-Nb}{2(2-b+N)}u(t_n)\right\|_{L^{\frac{4-2b}{N}+2}}^{\frac{4-2b}{N}+2}\right) = \rho_n^2 E[u_0].
\end{equation*}

Hence, $E[v_n] \rightarrow 0$ when $n\to \infty$ and moreover
\begin{equation}\label{connectionQv}
         \left\||\cdot|^\frac{-Nb}{2(2-b+N)}v_n\right\|_{L^{\frac{4-2b}{N}+2}}^{\frac{4-2b}{N}+2} \rightarrow \left(\frac{2-b+N}{N}\right)\|\nabla Q\|^2_{L^2}\text{ as }n\to \infty.
\end{equation}

Thus, recalling that $\sigma = \frac{2-b}{N}$, the sequence $\left\{v_n\right\}$ satisfies the hypotheses of Theorem \ref{teo1} with
\begin{equation*}
        M = \|\nabla Q\|_{L^2}, \,\,\, m^{2\sigma+2}= m^{\frac{4-2b}{N}+2} = \left(\frac{2-b+N}{N}\right)\|\nabla Q \|^2_{L^2} , \,\,\, C(M,m) = 1.
\end{equation*}

Therefore, there exists a sequence $\left\{x_n\right\} \subset \mathbb{R}^N$ and $V \in H^1(\mathbb{R}^N)$, with $\|V\|_{L^2} \geq \|Q\|_{L^2}$ such that, up to a subsequence,
\begin{equation}\label{weakV}
        \rho_n^\frac{N}{2} u(\rho_n\cdot + x_n, t_n) \rightharpoonup V \,\,\, \text{in }H^1.
\end{equation}

Using (\ref{weakV}), it follows, for every $A > 0$, that
\begin{equation*}
        \liminf_{n \rightarrow +\infty} \int_{|x|\leq A} \rho_n^N |u(\rho_n x +x_n,t_n)|^2 dx \geq \int_{|x| \leq A} |V|^2 dx.
\end{equation*}

Changing variables and using the hypothesis $\lambda(t_n)/\rho_n \rightarrow +\infty$, we obtain
\begin{equation*}
        \liminf_{n \rightarrow +\infty} \sup_{y \in \mathbb{R}^N} \int_{|x-y|\leq \lambda(t_n)} |u(x,t_n)|^2 dx \geq \int_{|x| \leq A} |V|^2 dx,
\end{equation*}

for every $A>0$, which gives
\begin{equation*}
        \liminf_{n \rightarrow +\infty} \sup_{y \in \mathbb{R}^N} \int_{|x-y|\leq \lambda(t_n)} |u(x,t_n)|^2 dx \geq \int_{\mathbb{R}^N} |V|^2 dx \geq C(M,m)\|Q\|_{L^2}.
\end{equation*}

Since $\left\{t_n\right\}$ is arbitrary, we get
\begin{equation*}
        \liminf_{t\nearrow T} \sup_{y \in \mathbb{R}^N} \int_{|x-y|\leq \lambda(t)} |u(x,t)|^2 dx \geq \int Q^2 dx.
\end{equation*}

Recalling that, for every $t \in [0, T)$ the function $y \mapsto \int_{|x-y|\leq \lambda(t)} |u(x,t)|^2 dx$ is continuous and goes to $0$ at infinity, we deduce the existence of $x(t)$ such that
\begin{equation}\label{contfun}
        \sup_{y \in \mathbb{R}^N} \int_{|x-y|\leq \lambda(t_n)} |u(x,t)|^2 dx = \int_{|x-x(t)|\leq \lambda(t_n)} |u(x,t)|^2 dx,
\end{equation}

which completes the proof of Theorem \ref{massconcentration}.
\end{proof}
\begin{re}\label{Rem1} In  Theorem \ref{massconcentration}, with the arguments in Weinstein \citep{W_1989}, if we suppose that $u$ is a radial solution, then we can give a more direct proof using a compact embedding. Indeed, let $\{t_n\}$ be such that $t_n\to T$ as $n\to\infty$, and set
$$v_n=\rho(t_n)^{\frac{N}{2}}u(t_n,\rho(t_n)x),$$ as done in the proof of Theorem \ref{massconcentration}.
Then,
$$E[v_n]=\rho(t_n)^{2}E[u_0]\to 0,\mbox{ as }n\to \infty,\,\,\,\|v_n\|_{L^2}=\|u_0\|_{L^2}\,\,\,\,\mbox{ and }\|\nabla v_n\|_{L^2}=\|\nabla Q\|_{L^2}.$$
Thus, by the compact embedding $H^1_{rad}(\Real^N)\hookrightarrow L^p(\Real^N)$, valid for $2<p<\frac{2N}{N-2}$, there exists $V\in L^p(\mathbb{R}^N)$ such that $v_n\to V$, up to a subsequence. Since $0<b<2$, then from the Holder's inequality, for all
$$\gamma\in\left(\frac{N}{N-b},\frac{2N^2}{(N-2)(4-2b+2N)}\right)$$
and $\gamma'$  such that $\displaystyle\frac{1}{\gamma'}+\frac{1}{\gamma}=1$, we have
$$\int |x|^{-b}|v_n-V|^{\frac{4-2b}{N}+2}\,dx\leq \left(\int |x|^{-b\gamma'}\,dx\right)^{\frac{1}{\gamma'}}\left(\int|v_n-V|^{\left(\frac{4-2b}{N}+2\right) \gamma}dx\right)^{\frac{1}{\gamma}}\to 0$$
as $n\to +\infty$.Therefore,
$$\int |x|^{-b}|V|^{\frac{4-2b}{N}+2}dx=\frac{2-b+N}{N}\|\nabla Q\|_{L^2}\neq 0.$$
By the sharp version of the Gagliardo-Nirenberg inequality, we have
$$0=\lim_{n\to \infty}E[v_n]=E[V]\geq \|\nabla V\|_{L^2}^2\left[1-\left(\frac{\|V\|_{L^2}}{\|Q\|_{L^2}}\right)^{\frac{4-2b}{N}}\right].$$
Hence, $\|V\|_{L^2}\geq \|Q\|_{L^2}$. The rest of the proof follows the same lines as in the proof of Theorem \ref{massconcentration}.
\end{re}

\section{Minimal mass solutions}\label{s_Minimal_Mass_Solutions}
In this section we give an alternative proof of Theorem \ref{caract}. We start with the following variational characterization of the ground state proved by Combet-Genoud \citep[Proposition 2]{Cmb}.
\begin{prop}\label{VarCar} Let $v \in H^1$ be such that
$$
\|v\|_{L^2} = \|Q\|_{L^2} \text{ and } E[v] = 0.
$$
Then there exist $\lambda_0 >0 $ and $\gamma_0 \in \Real$ such that $v(x) = e^{i \gamma_0}\lambda_0^{\frac{N}{2}}\psi(\lambda_0 x)$.
\end{prop}

Next, we prove the following proposition, which is a direct consequence of Theorem \ref{teo1} and Proposition \ref{VarCar}.
\begin{prop}\label{strongQ}
Let $u$ be a solution of \eqref{INLS} in the $L^2$-critical case ($\sigma=\frac{2-b}{N}$) blowing up in finite time $T>0$ and $\|u_0\|_{L^2}=\|Q\|_{L^2}$. Then, there exist functions $\rho(t)>0, \theta(t)\in \Real$ and $x(t)\in \Real^{N}$ such that
\begin{equation*}
\rho(t)^{N/2}\, e^{i\theta(t)}u(t,\rho(t)x)\to Q \mbox{ as }t\nearrow T,
\end{equation*}
strongly in $H^1(\mathbb{R}^N)$.
\end{prop}
\begin{proof}%
Using the same notation as the one in Theorem \ref{massconcentration}, we already have
$\|V\|_{L^2}\geq \|Q\|_{L^2}$. On the other hand, using the mass conservation \eqref{mass}, for $A>0$ we obtain
\begin{align*}
\|Q\|_{L^2}^2&=\|u_0\|_{L^2}^2=\|u(t)\|_{L^2}^2\geq \liminf_{t\to T}\int_{|x-x(t)|\leq A\rho(t)} |u(x,t)|^2\,dx\\&\geq\liminf_{{t\to T}}\int_{|x|\leq A}  \rho(t)^{N}|u(\rho(t)x+x(t),t)|^2\,dx\geq \int_{|x|\leq A}|V|^2\,dx.
\end{align*}
Since $A>0$ is arbitrary, we have $\|Q\|_{L^2}\geq \|V\|_{L^2}$. Thus, $\|Q\|_{L^2}= \|V\|_{L^2}$. Thus, the convergence in \eqref{weakV} is strong in $L^2(\mathbb{R}^N)$. Therefore, since $\{v_n\}$ (recall \eqref{defv} and \eqref{defv2}) is a bounded sequence in $H^1(\mathbb{R}^N)$, using the Gagliardo-Nirenberg inequality \eqref{gagliardo}, we have, as $n \to +\infty$,
$$\left\||\cdot|^{-\frac{Nb}{2(2-b+N)}}(v_n(\cdot+x_n)-V)\right\|_{L^{\frac{4-2b}{N}+2}}\leq C_{GN}\|\nabla v_n(\cdot+x_n)-\nabla V\|_{L^2}^2\|v_n(\cdot+x_n)-V\|_{L^2}^2\to 0.$$
Then,
$$|\cdot|^{-\frac{Nb}{2(2-b+N)}}(v_n(\cdot+x_n)-V)\to 0\,\,\,\mbox{ in }L^{\frac{4-2b}{N}+2}\left(\mathbb{R}^N\right).$$
Using again the Gagliardo-Nirenberg inequality \eqref{gagliardo} and the expression of $K_{opt}$ \eqref{gagliardo_opt} with $\sigma=\frac{2-b}{N}$,
\begin{align*}
\lim_{n\to \infty}\int |x|^{-b}|v_n(\cdot+x_n)|^{\frac{4-2b}{N}+2}\, dx&=\int |x|^{-b}|V|^{\frac{4-2b}{N}+2}\,dx\leq \frac{\frac{2-b}{N}+1}{\|Q\|_{L^2}^{\frac{4-2b}{N}}}\|\nabla V\|_{L^2}^{2}\|V\|_{L^2}^\frac{4-2b}{N}\\&=\left(\frac{2-b}{N}+1\right)\|\nabla V\|_{L^2}^{2}.
\end{align*}
Now, recalling (\ref{connectionQv}), we get $\|\nabla Q\|_{L^2}\leq \|\nabla V\|_{L^2}.$

Note that we also have
$$\|\nabla V\|_{L^2}\leq \limsup_{n\to \infty} \|\nabla v_n\|_{L^2}=\|\nabla Q\|_{L^2},$$
which gives us
\begin{equation}\label{VQ}
\|\nabla V\|_{L^2}=\|\nabla Q\|_{L^2}.
\end{equation}
Under these conditions, the convergence in \eqref{weakV} is strong in $H^1(\Real^N)$,
\begin{equation}\label{strongV}
        v_n(\cdot+x_n)\to V\mbox{ in }H^1(\mathbb{R}^N).
\end{equation}

Next we show that $\{x_n\}$ is bounded, which implies $x_n \to x^*$, up to a subsequence. To this end we apply the following Cauchy-Schwartz type inequality proved by Banica \citep{Bn}. We include the proof here for the sake of completeness.
\begin{lemma}\label{csinequality} Let $\theta$ be a real function. Then, for all $v\in H^{1}(\Real^N)$ satisfying $\|v\|_{L^2}\leq \|Q\|_{L^2}$, we have
\begin{equation}\label{BaIn}
\left|\int \mbox{Im }(v\nabla \overline{v})\nabla \theta\, dx\right|\leq \left(2E(v)\int|v|^2|\nabla\theta|^2\,dx\right)^{\frac{1}{2}}.
\end{equation}
\end{lemma}
\begin{proof}[Proof of Lemma \ref{csinequality}] By Gagliardo-Nirenberg Inequality \eqref{gagliardo_opt}, if $\|u\|_{L^2}\leq \|Q\|_{L^2}$, then
\begin{align*}
E[u]\geq \|\nabla u\|_{L^2}^2\left[1-\left(\frac{\|u\|_{L^2}}{\|Q\|_{L^2}}\right)^{\frac{4-2b}{N}}\right]\geq 0.
\end{align*}
Given $v\in H^1(\mathbb{R}^N)$ with $\|v\|_{L^2}\leq \|Q\|_{L^2}$ and $\theta$ a real function, we consider $u=e^{i\alpha\theta}v$ for some $\alpha\in \Real$, and write
\begin{equation*}
\|e^{i\alpha \theta}v\|_{L^2}=\|v\|_{L^2}\leq \|Q\|_{L^2}.
\end{equation*}
Therefore
\begin{equation*}
E[e^{i\alpha\theta }v]\geq 0, \mbox{ for all }\alpha\in \Real.
\end{equation*}
Next, we write the energy of $e^{i\alpha\theta}v$
\begin{align*}
E[e^{i\alpha\theta}v]&=\frac{1}{2}\int|i\alpha\nabla\theta+\nabla v|^2\,dx-\frac{N}{2N+4-2b}\int|x|^{-b}|v|^{\frac{4-2b}{|N|}+2}dx\nonumber\\
                    &=\frac{1}{2}\int \langle i\alpha\nabla\theta+\nabla v,-i\alpha\nabla\theta+\nabla \overline{v}\rangle\,dx-\frac{N}{2N+4-2b}\int |x|^{-b}|v|^{\frac{4-2b}{N}+2}dx\nonumber\\
                    &=\alpha^2\frac{1}{2}\int |v|^2|\nabla \theta|^2\,dx-\alpha\int \mbox{Im }(v\nabla \overline{v})\nabla \theta\,dx+E[v].
\end{align*}
Thus, we have a nonnegative quadratic function in $\alpha$, implying that its discriminant must be nonpositive, that is
$$\left(\int \mbox{Im }(v\nabla \overline{v})\nabla\,dx\right)^2-2E[v]\int|v|^2|\nabla \theta|^2\,dx\leq 0,$$
which implies \eqref{BaIn}.
\end{proof}

We now return to the proof of Proposition \ref{strongQ}. Note that, since $\|Q\|_{L^2} = \|V\|_{L^2}$ and by \eqref{strongV}, we get
\begin{equation}\label{blowupconverg}
|u(t_n,x)|^2\,dx-\|Q\|^2_{L^2}\delta_{x=x_n}\rightharpoonup 0,
\end{equation}
in the sense of tempered distributions. Suppose $x_n \to +\infty$ and let $\Phi\in C^{\infty}_0(\Real^N)$ be a nonnegative radial function such that
$$
\Phi(x)=|x|^2,\mbox{ if }|x|<1,\,\,\,\,\,\,\,\,|\nabla \Phi(x)|^2\leq C\Phi(x).
$$
We define
$$\Phi_p=p^2\Phi\left(\frac{x}{p}\right),\,\,\,\,\,\mbox{ and }g_p(t)=\int\Phi_p(x)|u(x,t)|^2\,dx.$$
Taking the derivative of $g_p$ and using the equation \eqref{INLS}, by the Lemma \ref{csinequality} we get
\begin{align}\label{g_p}
\nonumber|g'_p(t)|&=\left|\int \Phi_p(x)\left(\overline{u}\partial_tu+\overline{\overline{u}\partial_tu}\right)\,dx\right|\\
\nonumber&=\left|2\int \Phi_p(x)\mbox{Re } (\overline{u}\partial_tu)\,dx\right|=\left|2\int \Phi_p(x)\mbox{ Re}\left[\overline{u}\left(i\triangle u+i|x|^{-b}|u|^{2\sigma}u\right)\right]\,dx\right|\\
\nonumber&=\left|2\mbox{Im }\int \Phi_p(x) \left[\overline{u}\triangle u+|x|^{-b}|u|^{2\sigma+2}\,dx\right]\right|=\left|2\mbox{Im }\int \nabla(\Phi_p(x)\overline{u})\cdot\nabla u\,dx\right|\\
\nonumber&=\left|2\mbox{Im }\int (\nabla\Phi_p(x)\overline{u}+\Phi_p(x)\nabla\overline{u})\cdot\nabla u\,dx\right|=\left|2\mbox{Im }\int \nabla\Phi_p(x)\overline{u}\cdot\nabla u\,dx\right|\\
\nonumber&\leq\left(2E[u]\int |u|^2|\nabla\Phi_p(x)|^2\,dx\right)^{\frac{1}{2}}\leq C\left(2E[u]\int |u|^2\Phi_p(x)\,dx\right)^{\frac{1}{2}}\\
&=C\left(2E[u_0]g_p(t)\right)^{\frac{1}{2}}=C(u_0)\sqrt{g_p(t)}.
\end{align}
By integrating $\frac{g'_p(t)}{\sqrt{g_p(t)}}$, we have
\begin{equation}\label{gp}
|\sqrt{g_p(t)}-\sqrt{g_p(t_n)}|\leq C(u_0)|t-t_n|.
\end{equation}
The compact support of $g_p$ gives
$$g_p(t_n)\to 0\,\,\,\,\mbox{ as }n\to\infty\,\,\,\forall p\in \mathbb{N}.$$

Therefore, if $n\to\infty$ in \eqref{gp}, we get
$$|\sqrt{g_p(t)}|=\sqrt{g_p(t)}\leq C(u_0)|T-t|,$$
in other words,
$$g_p(t)\leq C(u_0)(T-t)^2, \,\,\,\,\,\forall\,p\in \mathbb{N}.$$
By monotone convergence, for any $t\in [0,T)$
\begin{equation}\label{cotagp}
\lim_{p\to \infty}g_p(t)=\lim_{p\to \infty}\int \Phi_p(x)|u(x,t)|^2\,dx=\int |x|^2|u(x,t)|^2\,dx\leq C(u_0)(T-t)^2.
\end{equation}
Then, by the uniform convergence to zero above and the convergence \eqref{blowupconverg}, we obtain
$$\|Q\|_{L^2}^2|x_n|^2-\int|x|^2|u(t_n,x)|^2\,dx\to 0,$$
when $n\to \infty$, and we conclude
$$|x_n|\leq \frac{1}{\|Q\|_{L^2}^2}C(u_0)(T-t_n)<\frac{1}{\|Q\|_{L^2}^2}C(u_0)T.$$
This implies that $x_n \to x^* \in \mathbb{R}^N$. Redefining V up to a translation (note that this does not
change the $L^2$ nor the $\dot{H}^1$ norm), we can assume
\begin{equation}\label{realstrongV}
        v_n \to V
\end{equation}
strongly in $H^1(\mathbb{R}^N)$.

Finally, we can characterize $V$. Calculating its energy yields
\begin{align*}
E(V)&=\frac{1}{2}\|\nabla V\|_{L^2}^2-\frac{1}{\frac{4-2b}{N}+2}\int |x|^{-b}|V|^{\frac{4-2b}{N}+2} \,dx\\
&=\lim_{n\to \infty}\left(\frac{1}{2}\|\nabla Q\|_{L^2}^2-\frac{1}{\frac{4-2b}{N}+2}\int |x|^{-b}|v_n|^{\frac{4-2b}{N}+2} \,dx\right)=0,
\end{align*}
where we used \eqref{connectionQv} and \eqref{VQ} in the last equality.
By Proposition \ref{VarCar}, there exists $\lambda_0>0$ and $\gamma_0\in \Real$ such that
$$V(x)=e^{i\gamma_0} \lambda_0^{\frac{N}{2}}Q(\lambda_0 x).$$
Hence,
$$\rho_n^{\frac{N}{2}}u(t_n,\rho_nx)\to e^{i\gamma_0} \lambda_0^{\frac{N}{2}}Q(\lambda_0 x)\mbox{ as }n\to \infty,$$
Since $\{t_n\}$ is an arbitrary sequence, we have that for all $t\in \Real$ there exists $\rho(t), \theta(t)$ and $x(t)$ such that
\begin{equation}\label{delta_0}\rho(t)^{\frac{N}{2}}e^{i\gamma(t)}u(t,\rho(t)x)\to  Q,\mbox{ as }t\nearrow T,\end{equation}
which completes the proof.
\end{proof}

We are now able to prove the main theorem of this section.
\begin{proof}[Proof of Theorem \ref{caract}]

By \eqref{delta_0}, we can write
\begin{equation}\label{blowupconverg0}
|u(t_n,x)|^2\,dx-\|Q\|^2_{L^2}\delta_{x=0}\rightharpoonup 0,
\end{equation}
in the sense of distributions. By the same calculations in \eqref{g_p} and \eqref{gp}, and noting that \eqref{blowupconverg0} implies
$g_p(t_n)\to 0\,\,\,\,\mbox{ as }n\to\infty\,\,\,\forall p\in \mathbb{N} $, we conclude

$$\int |x|^2|u(x,t)|^2\,dx\leq C(u_0)(T-t)^2.$$
By virial identity,
it follows that
\begin{equation*}
\int|x|^2|u(x,t)|^2\,dx=8t^2E[u_0]-4t\Img\int (x) u_0\nabla\overline{u}_0\,dx+\int |x|^2|u_0|^2\, dx
\end{equation*}
which gives
\begin{align*}
E\left[e^{i\frac{|x|^2}{4t}}u_0\right]&=\frac{1}{2}\int\left|\nabla\left(e^{i\frac{|x|^2}{4t}}u_0\right)\right|^2\,dx-\frac{N}{2N+4-2b}\int\left|e^{i\frac{|x|^2}{4t}}u_0\right|^{\frac{4-2b}{N}+2}\,dx\nonumber\\
                                          &=\frac{1}{8t^2}\int|x|^2|u_0|^2\,dx-\frac{1}{2t}\Img\int u_0 x \cdot\nabla u_0\,dx+E[u_0]\nonumber\\
                                          &=\frac{1}{8t^2}\left(\int |x|^2|u(x,t)|^2\,dx-8t^2E[u_0]+4t\Img\int u_0 x \cdot\nabla\overline{u}_0\,dx\right)\nonumber\\
                                          &\qquad\qquad\,\,\,\,\,-\frac{1}{2t}\Img\int u_0 x \cdot\nabla u_0\,dx+E[u_0]\nonumber\\
                                          &=\frac{1}{8t^2}\int |x|^2|u(x,t)|^2\,dx\leq \frac{C(u_0)}{8t^2}(T-t)^2,
\end{align*}
where we used \eqref{cotagp} in the last inequality. By making $t\to T$, we get $E\left[e^{i\frac{|x|^2}{4T}}u_0\right]=0$ and also $\left\|e^{i\frac{|x|^2}{4T}}u_0\right\|_{L^2}=\|Q\|_{L^2}$. Therefore, by Proposition \ref{VarCar}, there exist $\lambda_0>0$ and $\gamma_0\in \Real$ such that
$$e^{i\frac{|x|^2}{4T}}u_0(x)=e^{i\gamma_0}\lambda_0^{\frac{N}{2}}Q(\lambda_0x).$$
\end{proof}

\section{Critical norm concentration}\label{s_Critical_Norm_Concentration}

This section is devoted to the $L^{\sigma_c}(\mathbb{R}^N)$ norm concentration for the intercritical INLS equation. We will need the so-called Strauss lemma for radial $H^1(\mathbb{R}^N)$ functions.
\begin{lemma}[Strauss]
Let $N \geq 2$. If $u \in H^1_{rad}(\mathbb{R}^N)$, then for any $R>0$\begin{equation}\label{Strauss}
\|u\|_{L^\infty_{(|x|\geq R)}}
\leq  \frac{1}{R^{\frac{N-1}{2}}}\|u\|^\frac{1}{2}_{L^2_{\{|x|\geq R\}}}\|\nabla u\|^\frac{1}{2}_{L^2_{\{|x|\geq R\}}}.
\end{equation}
\end{lemma}
\begin{proof}
Since $u$ is radial we deduce
\begin{eqnarray*}
\sup_{|x|\geq R}|u(x)|^2&=&\sup_{|x|\geq R}\frac{1}{2}\int_{|x|}^{+\infty} \partial_r(u^2)dr\\
&\leq & \int_{R}^{+\infty} u \partial_rudr\\
&\leq & \left(\int_{R}^{+\infty} |u|^2dr \right)^{\frac{1}{2}}\left(\int_{R}^{+\infty} |\partial_r u|^2dr \right)^{\frac{1}{2}},
\end{eqnarray*}
where we have used that $u$ has to vanish at infinity and the Cauchy-Schwarz inequality.  On the other hand, the fact that $|x|\geq R$ (or $r\geq R$) implies $1 \leq \frac{r}{R}$ so
\begin{eqnarray*}
\sup_{|x|\geq R}|u(x)|^2&\leq & \left(\int_{R}^{+\infty} |u|^2 \left(\frac{r}{R}\right)^{N-1} dr\right)^{\frac{1}{2}}\left(\int_{R}^{+\infty} |\partial_ru|^2\left(\frac{r}{R}\right)^{N-1}dr \right)^{\frac{1}{2}}\\
&\leq & \frac{1}{R^{\frac{N-1}{2}}} \left(\int_{R}^{+\infty} |u|^2 r^{2(N-1)} dr\right)^{\frac{1}{2}}\frac{1}{R^{\frac{N-1}{2}}}\left(\int_{R}^{+\infty} |\partial_ru|^2 r^{2(N-1)}dr \right)^{\frac{1}{2}}\\
&=&  \frac{1}{R^{N-1}} \left(\int_{R}^{+\infty} |u|^2 dx \right)^{\frac{1}{2}}\left(\int_{R}^{+\infty} |\nabla u|^2 dx \right)^{\frac{1}{2}}\\
&\leq& \frac{1}{R^{N-1}}\|u\|_{L^2_{\{|x|\geq R\}}}\|\nabla u\|_{L^2_{\{|x|\geq R\}}},
\end{eqnarray*}
where in the third line we have used the fact that $|\partial_r u|=|\nabla u|$ for radial functions. We finish the proof by taking the square root on both sides.
\end{proof}
By using  \eqref{Strauss} and Young's inequality, one also proves an estimate for the $L^{2\sigma+2}(\mathbb{R}^N)$ norm outside the ball.
\begin{cor}[Radial Gagliardo-Nirenberg]Let $N \geq 2$. If $u \in H^1(\mathbb{R}^N)_{rad}$, then for any $R>0$, $\eta>0$ and $\sigma < 2$,
\begin{equation}\label{GN_radial}
\int_{|x|\geq R}|x|^{-b}|u|^{2\sigma+2}\,dx \leq \eta  \|\nabla u\|_{L^2_{\{|x|\geq R\}}}^{2} + \frac{C(\eta)}{R^{\frac{2(\sigma(N-1)+b)}{2-\sigma}}}  \|u\|_{L^2_{\{|x|\geq R\}}}^{\frac{2(\sigma+2)}{2-\sigma}}.
\end{equation}
\end{cor}
\begin{proof} If $\sigma < 2$, then $\frac{2}{\sigma} < 1$, and we have 
\begin{align*}
\int_{|x|\geq R} |x|^{-b}|u|^{2\sigma+2} \, dx&\leq R^{-b}\|u\|_{L^\infty_{\{|x|\geq R\}}}^{2\sigma}\|u\|_{L^2_{\{|x|\geq R\}}}^2\\
& \leq R^{-(\sigma(N-1)+b)}\|\nabla u\|_{L^2_{\{|x|\geq R\}}}^{\sigma}\|u\|_{L^2_{\{|x|\geq R\}}}^{\sigma + 2}\\
& \leq \eta  \left(\|\nabla u\|_{L^2_{\{|x|\geq R\}}}^{\sigma}\right)^\frac{2}{\sigma} + C(\eta) \left(R^{-(\sigma(N-1)+b)} \|u\|_{L^2_{\{|x|\geq R\}}}^{\sigma + 2}\right)^{\frac{2}{2-\sigma}}\\
& = \eta  \|\nabla u\|_{L^2_{\{|x|\geq R\}}}^{2} + \frac{C(\eta)}{R^{\frac{2(\sigma(N-1)+b)}{2-\sigma}}}  \|u\|_{L^2_{\{|x|\geq R\}}}^{\frac{2(\sigma+2)}{2-\sigma}}.
\end{align*}
\end{proof}

We will also need the $L^{\sigma_c}(\mathbb{R}^N)$ version of the Gagliardo-Nirenberg inequality, valid for any $H^1(\mathbb{R}^N)$ (not necessarily radial) function.
\begin{lemma}[Critical Gagliardo-Nirenberg]\label{GNcrit2}
Let $N\geq 2$, $0<b<2$ and $\sigma_c = \frac{2N\sigma}{2-b}$. If $u \in H^1(\mathbb{R}^N)$ and $\frac{2-b}{N}<\sigma<\sigma^*_b$ (recall \eqref{sigma_*}), then
\begin{equation}\label{GNcrit}
\int |x|^{-b}|u|^{2\sigma+2} \, dx\leq c \|\nabla u\|^2_{L^2}\|u\|^{2\sigma}_{L^{\sigma_c}}.
\end{equation}
\end{lemma}

To prove this inequality, we will assume the following version of the Sobolev embedding (see Stein-Weiss \citep[Theorem B*]{stein}).
\begin{lemma}
Let $1 < p \leq q <+\infty$, $N\geq 1$, $0 < s < N$ and $\beta \geq 0$ satistfy the conditions
$$
\beta < N/q, \quad \quad s= \frac{N}{p}-\frac{N}{q}+\beta.
$$
%$$
%s= \frac{N}{p}-\frac{N}{q}+\beta.
%$$
Then, for any $u \in W^{s,p}$ we have
$$
\left\| |x|^{-\beta}u\right\|_{L^{q}} \leq c (\beta, p, q, N,s) \left\||\nabla|^su\right\|_{L^p}.
$$
\end{lemma}
\begin{proof}[Proof of Lemma \ref{GNcrit2}]
Using Holder's inequality, we have
$$
\int |x|^{-b}|u|^{2\sigma+2} \, dx\leq c \||x|^{-b}|u|^2\|_{L^\frac{N}{N-2+b}}\||u|^{2\sigma}\|_{L^\frac{N}{2-b}}
= c \| |x|^{-b/2}u\|^2_{L^\frac{2N}{N-2+b}} \|u\|^{2\sigma}_{L^{\sigma_c}}.
$$
By Sobolev inequality, $\||x|^{-b/2}u\|_{L^\frac{2N}{N-2+b}} \leq c \|\nabla u\|_{L^2}$. Thus,
$$
\int |x|^{-b}|u|^{2\sigma+2}\, dx \leq c \|\nabla u\|^2_{L^2}\|u\|^{\sigma}_{L^{\sigma_c}}.
$$
\end{proof}
Now, let $\phi \in C^{\infty}_0(\Real^N)$ be a positive radial cut-off solution such that
$$
\phi(x)=
\left\{
\begin{array}{ll}
1&,\,\mbox{ if } |x|\leq 1,\\
0&,\,\mbox{ if } |x|\geq 2.
\end{array}
\right.
$$
Let $u$ be a $H^1(\Real^N)$-solution of \eqref{INLS}. For $R=R(t)$, to be chosen later, define the inner and outer spatial localizations of $u(x,t)$ at radius $R(t)>0$ as
$$u_1(x,t)=\phi\left(\frac{x}{R}\right)u(x,t)\,\,\,\,\mbox{ and }\,\,\,\,u_2(x,t)=\left(1-\phi\left(\frac{x}{R}\right)\right)u(x,t
).$$
Let $\chi \in C^{\infty}_0(\Real^N)$ be a radial function so that $\widehat{\chi}(0)=1$ and $\chi(x)=0$ for $|x|\geq 1$. For $\rho=\rho(t)$, define the inner and outer frequency localizations at radius $\rho(t)$ of $u_1$ as $$\widehat{u}_{1L}(\xi,t)=\widehat{\chi}\left(\frac{\xi}{\rho}\right)\widehat{u}_1(\xi,t)\,\,\,\,\mbox{ and }\,\,\,\,\widehat{u}_{1H}(\xi,t)=\left(1-\widehat{\chi}\left(\frac{x}{\rho}\right)\right)\widehat{u}_1(\xi,t).$$
\begin{prop}\label{Plpc} Let $u$ be a radial $H^1(\Real^N)$-solution of \eqref{INLS}, which blows up in finite time $T>0$. There exist positive constants $c_1$ and $c_2$ such that the following holds true. Define $$R(t)=c_1\frac{\|u_0\|_{L^2}^{\frac{2(\sigma+2)}{2(\sigma(N-1)+b)}}}{\|\nabla u\|^{\frac{2(2-\sigma)}{2(\sigma(N-1)+b)}}_{L^2}},\,\,\,\,\,\,\,\mbox{    }\,\,\,\,\,\,\,\rho(t)=c_2\|\nabla u(t)\|_{L^2}^{\frac{1}{1-s_c}},$$
and consider the decomposition $u=u_{1L}+u_{1H}+u_2$ described above.
Then,
\begin{itemize}
\item[(i)] There exist an absolute constant $c>0$ such that
\begin{equation}\label{lmtabove}
\left\|u_{1L}\right\|_{L^{\sigma_c}}\geq c\,\,\,\,\mbox{ as } t\to T.
\end{equation}
\item[(ii)] Suppose that there are constants $c^*$ and $\gamma > \sigma_c$ such that $\|u_1(t)\|_{L^{\gamma}}\leq c^*$ for all $t$ close enough to $T$. Then,
\begin{equation}\label{itemii}
\left\|u_{1}\right\|_{L^{\sigma_c}\left(|x-x_0(t)|\leq \rho(t)^{-1}\right)}\geq \frac{c}{(c^*)^{\frac{\sigma_c(N-b)}{N(2\sigma+2)-\sigma_c(N-b)}}}\,\,\,\,\mbox{ as } t\to T
\end{equation}
for some constant $c>0$, where $x_0(t)$ is a positive function such that $$|x_0(t)|\,\rho(t)\leq c(c^*)^{\frac{\sigma_cN(2\sigma+2)}{(N-1)(N(2\sigma+2)-\sigma_c(N-b))}}.$$
\end{itemize}
\end{prop}
\begin{proof}
If $\|\nabla u(t)\|_{L^2}\to \infty$, as $t\to T$, then $$\lim_{t\to T}\frac{\left\||x|^{-\frac{b}{2\sigma+2}}u(t)\right\|_{L^{2\sigma+2}}^{2\sigma+2}}{\|\nabla u(t)\|^2_{L^2}}=\sigma + 1.$$ Hence, for $t$ close to $T$, we have \begin{align}\label{equ}
\|\nabla u(t)\|_{2}^{2}&\leq \left\||x|^{-\frac{b}{2\sigma+2}}u(t)\right\|_{L^{2\sigma+2}}\nonumber\\
&\leq c\left\||x|^{-\frac{b}{2\sigma+2}}u_{1L}(t)\right\|_{L^{2\sigma+2}}^{2\sigma+2}+c\left\||x|^{-\frac{b}{2\sigma+2}}u_{1H}(t)\right\|_{L^{2\sigma+2}}^{2\sigma+2}+c\left\||x|^{-\frac{b}{2\sigma+2}}u_{2}(t)\right\|_{L^{2\sigma+2}}^{2\sigma+2}.
\end{align}
Now, using the radial Gagliardo-Nirenberg  \eqref{GN_radial},  choosing $c_1$ small enough, we obtain
\begin{align}\label{eq_48}
c\left\||x|^{-\frac{b}{2\sigma+2}}u_{2}(t)\right\|_{L^{2\sigma+2}}^{2\sigma+2}&=c\left\||x|^{-\frac{b}{2\sigma+2}}\left(1-\phi\left(\frac{x}{R(t)}\right)\right)u(t)\right\|_{L^{2\sigma+2}}^{2\sigma+2}\nonumber\\&\leq c \int_{|x|\leq R(t)}|x|^{-b}|u(x,t)|^{2\sigma+2}\,dx\leq \frac{1}{8}\|\nabla u(t)\|_{L^2}^2+\widetilde{c}\frac{\|u_0\|_{L^2}^{\frac{2(\sigma+2)}{2-\sigma}}}{R(t)^{\frac{2(\sigma(N-1)+b)}{2-\sigma}}}
\nonumber\\&\leq \frac{\tilde{c}\,c_1}{4}\|\nabla u\|_{L^2}^2 \leq \frac{1}{4}\|\nabla u\|_{L^2}^2.
\end{align}
On the other hand, invoking the Gagliardo-Nirenberg inequality \eqref{GNcrit} and from Sobolev embedding, we get
\begin{align*}
c\left\||x|^{-\frac{b}{2\sigma+2}}u_{1H}\right\|^{L^{2\sigma+2}}_{2\sigma+2}&\leq c \|\nabla u\|_{L^2}^2\|u_{1H}\|_{L^{\sigma_c}}^{2\sigma}\leq c\|\nabla u\|_{L^2}^2\|u_{1H}\|_{\dot{H}^{s_c}}^{2\sigma}\nonumber\\
&=c\|\nabla u\|_{L^2}^2\left\||\xi|^{s_c}\widehat{u}_{1H}\right\|_{L^2}^{2\sigma}=c\left\||\xi|^{s_c}\left(1-\widehat{\chi}\left(\frac{\xi}{\rho(t)}\right)\right)\widehat{u}_1\right\|_{L^2}^{2}
\end{align*}
By the mean value theorem
$$|1-\widehat{\chi}(\xi)|\leq C\min\{1,|\xi|\}.$$
If $|\xi|\leq \rho$, then
$$|\xi|^{s_c}\left|1-\widehat{\chi}\left(\frac{\xi}{\rho}\right)\right|\leq |\xi|^{s_c}\frac{|\xi|}{\rho}\leq \frac{|\xi|}{\rho^{1-s_c}}.$$
On the other hand, if $|\xi|\geq \rho$, it is easy to see that
$$|\xi|^{s_c}\left|1-\widehat{\chi}\left(\frac{\xi}{\rho}\right)\right|\leq |\xi|^{s_c}= |\xi|^{s_c}\frac{|\xi|}{|\xi|}\leq \frac{|\xi|}{\rho^{1-s_c}}.$$
Hence, from the definition of $\rho$ and choosing $c_2$ big enough, we have
\begin{align}\label{eq_410}
c\left\||x|^{-\frac{b}{2\sigma+2}}u_{1H}\right\|_{L^{2\sigma+2}}^{2\sigma+2}&\leq c\left\|\frac{|\xi|}{\rho(t)^{1-s_c}}\widehat{u}_{1H}\right\|_{L^2}^{2\sigma}\leq c\frac{\|\nabla u\|_{L^2}^{2\sigma+2}}{\rho(t)^{2\sigma(1-s_c)}}\nonumber\\&\leq \frac{1}{4}\|\nabla u\|_{L^2}^2.
\end{align}
Therefore, in view of the inequality \eqref{equ}, \eqref{eq_48} and \eqref{eq_410}, we deduce
\begin{equation}\label{inq}
\left\|\nabla u\right\|_{L^{2}}^{2}\leq c\left\||x|^{-\frac{b}{2\sigma+2}}u_{1L}\right\|_{L^{2\sigma+2}}^{2\sigma+2}.
\end{equation}
Again using the Gagliardo-Nirenberg inequality and the above inequality, we obtain
\begin{align*}
c\left\||x|^{-\frac{b}{2\sigma+2}}u_{1L}\right\|_{L^{2\sigma+2}}^{2\sigma+2}\leq c\|\nabla u\|_{L^2}^{2}\|u_{1L}\|_{L^{\sigma_c}}^{2\sigma}\leq 2c\left\||x|^{-\frac{b}{2\sigma+2}}u_{1L}\right\|_{L^{2\sigma+2}}^{2\sigma+2}\|u_{1L}\|_{L^{\sigma_c}}^{2\sigma},
\end{align*}
that is,
$$\|u_{1L}\|_{L^{\sigma_c}}^{2\sigma}\geq \frac{1}{2c},$$
and \eqref{lmtabove} is proved.

Suppose now that there exist positive constants $c^*$ and $\gamma > \sigma_c$ such that $\|u_1(t)\|_{L^{\gamma}}\leq c^*$ for all $t\in [0,T)$. Applying \eqref{inq}, the Hölder inequality and by same argument that we have used in \eqref{contfun},
we get\begin{align*}
\|\nabla u(t)\|_{L^2}^2&\leq \left\||x|^{-\frac{b}{2\sigma+2}}u_{1L}\right\|_{L^{2\sigma+2}}^{2\sigma+2}=\int |x|^{-b}|u_{1L}(t)|^{2\sigma+2}\,dx\nonumber\\
&\leq \left(\int |x|^{-bp'}\,dx\right)^{\frac{1}{p'}}\left(\int |u_{1L}(x,t)|^{p(2\sigma+2)}\,dx\right)^{\frac{1}{p}}\nonumber\\
&\leq C_b\left(\int |u_{1L}|^{\gamma}\,dx\right)^{\frac{1}{p}}\|u(t)\|_{L^{\infty}}^{2\sigma+2-\frac{\gamma}{p}}\nonumber\\
&\leq C_b(c^*)^{\frac{\gamma}{p}}\sup_{x\in \Real^N}\left|\int \rho^N\chi(\rho(x-y))u(y,t)\,dy\right|^{2\sigma+2-\frac{\gamma}{p}}\nonumber\\
&\leq C_b(c^*)^{\frac{\gamma}{p}}\rho^{N\left(2\sigma+2-\frac{\gamma}{p}\right)}\left[\int_{|x-x_0(t)|\leq \rho(t)^{-1}}|u(x,t)|\,dx\right]^{2\sigma+2-\frac{\gamma}{p}}\nonumber\\
&\leq \omega_N^{\left(\frac{2N\sigma-2+b}{2N\sigma}\right)\left(2\sigma+2-\frac{\gamma}{p}\right)}C_b(c^*)^{\frac{\gamma}{p}}\rho^{N\left(2\sigma+2-\frac{\gamma}{p}\right)}\rho^{-N\left(\frac{2N\sigma-2+b}{2N\sigma}\right)\left(2\sigma+2-\frac{\gamma}{p}\right)}\|u\|^{2\sigma+2-\frac{\gamma}{p}}_{L^{\sigma_c}(|x-x_0(t)|\leq \rho(t)^{-1})}\nonumber\\
&=\omega_N^{\left(\frac{2N\sigma-2+b}{2N\sigma}\right)\left(2\sigma+2-\frac{\gamma}{p}\right)}C_b(c^*)^{\frac{\gamma}{p}}\rho^{\frac{2-b}{2\sigma}\left(2\sigma+2-\frac{\gamma}{p}\right)}\|u(t)\|_{L^{\sigma_c}(|x-x_0(t)|\leq \rho(t)^{-1})}^{2\sigma+2-\frac{\gamma}{p}},
\end{align*}
where $\omega_N$ is the volume of the unit ball in $\Real^{N}$, $C_b=\left\||\cdot|^{-b}\right\|_{L^{p'}}$ and $(p,p')$ is chosen so that $p \geq 1$, $\frac{1}{p'}+\frac{1}{p}=1$ and $bp'<N$. If $p=\frac{N}{N-b}\frac{\gamma}{\sigma_c}$, then it satisfies these conditions and still
$$
\frac{2-b}{2\sigma(1-s_c)}\left(2\sigma+2-\frac{\gamma}{p}\right)=2.
$$
Hence, we have the following estimates
\begin{equation*}
\|\nabla u(t)\|_{L^2}^2\leq \omega_NC_b(c^*)^{\frac{\sigma_c(N-b)}{N}}\|\nabla u\|^{2}_{L^2}\|u(t)\|_{L^{\sigma_c}(|x-x_0(t)|\leq\rho(t)^{-1})}^{2\sigma+2-\frac{\sigma_c(N-b)}{N}},
\end{equation*}
or more precisely,
$$
\|u(t)\|_{L^{\sigma_c}(|x-x_0(t)|\leq \rho(t)^{-1})}\geq \frac{c}{(c^*)^{\frac{\sigma_c(N-b)}{N(2\sigma+2)-\sigma_c(N-b)}}}.
$$
To complete the proof, suppose that there exist $t_n \to T$ such that $$|x_0(t_n)|\rho(t_n)\gg (c^*)^{\frac{\sigma_cN(2\sigma+2)}{(N-1)(N(2\sigma+2)-\sigma_c(N-b))}}.$$ Consider the annular region
$$A:=\{x\in \Real^N;\, |x_0|-\rho^{-1}\leq |x|\leq |x_0|+\rho^{-1}\}$$
that contains  $C_N\frac{|x_0|^{N-1}}{(\rho^{-1})^{N-1}}$ disjoint balls, each of radius $\rho^{-1}$, centered on the sphere at radius $|x_0|$. By the radiality assumption, on each ball $B$, we have
$$\|u_1(t_n)\|_{L^{\sigma_c}(A)}^{\sigma_c}\geq C(|x_0|\rho)^{N-1} \|u_1(t_n)\|_{L^{\sigma}(B)}\geq \frac{CM^{\sigma_c}}{(c^*)^{\frac{\sigma_c^2(N-b)}{N(2\sigma+2)-\sigma_c(N-b)}}}(|x_0|\rho)^{N-1}\gg (c^*)^{\sigma_c},$$
which contradicts the assumption $\|u_1(t)\|\leq c^*$. This completes the proof.
\end{proof}
Finally, we prove our main result of this section.
\begin{proof}[Proof of Theorem \ref{Tlpc}]
Note that from item (i) of the Proposition \ref{Plpc} and Young's inequality, we have
$$c\leq \|u_{1L}(t)\|_{L^{\sigma_c}}=\|\rho^N\chi(\rho\cdot)\ast u_1(t)\|_{L^{\sigma_c}}\leq c\|u_1(t)\|_{L^{\sigma_c}}.$$
So, it follows that $\|u_1(t)\|_{L^{\sigma_c}}$ is bounded from below. We now divide the proof into two cases.
\begin{itemize}
\item[Case 1.] If $\|u_1(t)\|_{L^{\sigma_c}}$ is not bounded (so $\| u_1(t)\|_{L^{\gamma}} $ is also unbounded for any $\gamma > \sigma_c$), then there exist a sequence $t_n\to T$ such that $\|u_1(t_n)\|_{L^{\sigma_c}}\to +\infty.$

Since
$$\|u_1(t)\|_{L^{\sigma_c}}^{\sigma_c}=\left\|\phi\left(\frac{x}{R}\right)u(t)\right\|_{L^{\sigma_c}}^{\sigma_c}\leq \int_{|x|\leq2R}|u(t)|^{\sigma_c}\,dx=\|u(t)\|_{L^{\sigma_c}_{\{|x|\leq 2R\}} },$$
 we have \eqref{inft}.
\item[Case 2.] If, on the other hand, $\| u_1(t)\|_{L^{\gamma}} \leq c^*$, for some $\gamma > \sigma_c$, then we have \eqref{itemii}. Since $|x_0(t)|\,\rho(t)\leq c(c^*)^{\frac{\sigma_cN(2\sigma+2)}{(N-1)(N(2\sigma+2)-\sigma_c(N-b))}}$, then$$
\frac{c}{(c^*)^{\frac{\sigma_c(N-b)}{N(2\sigma+2)-\sigma_c(N-b)}}}
\leq \left\|u_{1}\right\|_{L^{\sigma_c}\left(|x-x_0(t)|\leq \rho(t)^{-1}\right)}
\leq \left\|u_{1}\right\|_{L^{\sigma_c}\left(|x|\leq c(c^*)^{\frac{\sigma_cN(2\sigma+2)}{(N-1)(N(2\sigma+2)-\sigma_c(N-b))}} \rho(t)^{-1}\right)},
$$
which gives \eqref{fint}.
\end{itemize}
\end{proof}

\section{Appendix}
In this appendix we prove Theorem \ref{teo1}. First we recall a result obtained by Hmidi-Keraani \citep{keraani}.
\begin{teo}\label{profile}
Define $2^* = \infty$ if $N = 1,2$ and $2^* = \frac{2N}{N-2}$ if $N \geq 3$. Let $\left\{v_n\right\}$ be a bounded sequence in $H^1(\mathbb{R}^N)$. Then, there exists a subsequence of $\left\{v_n\right\}$ (also denoted $\left\{v_n\right\}$), a family $\left\{x^j\right\}$ of sequences in $\mathbb{R}^N$ and a sequence $\left\{V^j\right\}$ of $H^1(\mathbb{R}^N)$ functions such that
\begin{enumerate}
        \item[(i)] for every $k \neq j$, $|x_n^k - x_n^j| \xrightarrow[n \rightarrow +\infty]{} 0$
        \item[(ii)] for every $l \geq 1$ and every $x \in \mathbb{R}^N$,
\end{enumerate}

\begin{equation*}
        v_n(x) = \sum_{j=1}^l   V^j(x-x_n^j) + v^l_n(x),
\end{equation*}

with
\begin{equation}\label{ortho1}
        \limsup_{n \rightarrow +\infty} \|v_n^l\|_{L^p} \xrightarrow[l \rightarrow +\infty]{} 0,
\end{equation}
for every $p \in (2,2^*)$.

Moreover, as $n \rightarrow +\infty$,
\begin{equation}\label{ortho2}
        \|v_n\|_{L^2}^2 = \sum_{j=1}^l \|V^j\|_{L^2}^2 + \|v_n^l\|_{L^2}^2 + o(1),
\end{equation}

\begin{equation}\label{ortho3}
        \|\nabla v_n\|_{L^2}^2 = \sum_{j=1}^l \|\nabla V^j\|_{L^2}^2 + \|\nabla v_n^l\|_{L^2}^2 + o(1).
\end{equation}
\end{teo}
\begin{proof}[Proof of the Theorem \ref{teo1}]
In order to use Theorem \ref{profile} in the INLS setting, we need to control the quantity $\int |x|^{-b} |v_n^l|^{2\sigma+2}$. Indeed, assume first that $N \geq 3$. Recalling the condition $\sigma < \frac{2-b}{N-2}$, we get $\frac{N-b}{N} - (\sigma + 1)\frac{N-2}{N} > 0$. Letting $0 < \epsilon < \frac{N-b}{N} - (\sigma + 1)\frac{N-2}{N}$, and choosing $\gamma$ such that $\frac{1}{\gamma} = \frac{b}{N} + \epsilon$, we conclude that $\gamma < \frac{N}{b}$ and $2 < (2\sigma +2)\gamma' < 2^*$ (where $\gamma'$ is such that $\frac{1}{\gamma} + \frac{1}{\gamma'} = 1$). Thus, by the Hölder inequality, we have
\begin{equation*}
        \int_{|x|\leq 1} |x|^{-b} |v_n^l|^{2\sigma+2} \leq \left(\int_{|x|\leq 1} |x|^{-b \gamma}\right)^\frac{1}{\gamma} \left(\int|v_n^l|^{(2\sigma+2)\gamma'} \right)^\frac{1}{\gamma'} \leq C(b,N) \|v_n^l\|_{L^{(2\sigma+2)\gamma'}}^{2\sigma+2}.
\end{equation*}

We conclude that $\displaystyle \int |x|^{-b} |v_n^l|^{2\sigma+2} \leq C(b,N) \|v_n^l\|_{L^{(2\sigma+2)\gamma'}}^{2\sigma+2}+ \int_{|x|> 1} |v_n^l|^{2\sigma+2}$ and therefore, for $2 < 2\sigma +2 < 2^*$
\begin{equation}\label{ortho4}
        \displaystyle\limsup_{n \rightarrow +\infty} \int |x|^{-b} |v_n^l|^{2\sigma+2} \xrightarrow[l \rightarrow +\infty]{} 0.
\end{equation}
The case $N= 2$ is similar, in fact, easier since $2^*=\infty$ and we omit the details.

%%%%%%%%%%%%%%%%%%%%%%%%%%%%%%%%%%%%%%%%%%%%%%%%%%%%%%%%%%%%
%%Useful calculations (g = \gamma, s = \sigma, p = 2s+2)
\iffalse
g < N/b
1/g > b/N
1/g = b/N + e

%Checking pg' < 2*
1/g' = 1-1/g = 1 - b/N - e
g' = N/(N-b-Ne)
pg' = pN/(N-b-Ne) < 2N/(N-2)
2N-2b-2Ne > Np-2p
e < (N-b)/N-p(N-2)/2N

(N-b)/N-p(N-2)/2N > 0
(N-b) > p(N-2)/2
p < 2(N-b)/(N-2)

(obs:
2sigma+2 < 2(N-b)/(N-2)
sigma < (2-b)/(N-2))

%Checking pg' > 2
pN/(N-b-Ne) > 2
pN > 2N - 2b - 2Ne
p > 2 - 2b/N - 2e
sigma > - b/N - e   ok!

%N sei pq fiz isso
2(N-b)>N-2
N - 2b > -2
2b < N+2
b < N/2 + 1 ok!, pois N/2+1 >= min{N,2}

%Só útil no caso crítico
(4-2b)/N +2 < 2(N-b)/(N-2)
4N-8 -2(N-2)b +2N(N-2)< 2N^2 - 2Nb
4N-8 - 2N^2 + 2N^2 - 4N< 2Nb -4b - 2Nb
4b < 8
b < 2 ok!
\fi
%%%%%%%%%%%%%%%%%%%%%%%%%%%%%%%%%%%%%%%%%%%%%%%%%%%%%%%%%%%%%

We are now able to prove Theorem \ref{teo1}. We can extract subsequences and use $\lim$ instead of $\limsup$. According to Theorem \ref{profile} and the discussion above, the sequence $\left\{v_n\right\}$ can be written, up to a subsequence, for every $l\geq 1$, as
\begin{equation*}
        v_n(x) = \sum_{j=1}^l   V^j(x-x_n^j) + v^l_n(x),
\end{equation*}
such that (\ref{ortho1}), (\ref{ortho2}) and (\ref{ortho4}) hold. So
\begin{equation*}
        m^{2\sigma+2} \leq \limsup_{n \rightarrow +\infty} \left\||\cdot|^\frac{-b}{2\sigma+2} v_n\right\|_{L^{2\sigma+2}}^{2\sigma+2} = \limsup_{n \rightarrow +\infty} \left\| \sum_{j=1}^l|\cdot|^\frac{-b}{2\sigma+2} V^j(\cdot - x_n^j)\right\|_{L^{2\sigma+2}}^{2\sigma+2}.
\end{equation*}

Using the elementary inequality (see (1.10) in Gérard \citep{G_98})
\begin{equation*}
        \left|\left|\sum_{j=1}^l a_j \right|^{2\sigma + 2} - \sum_{j=1}^l|a_j|^{2\sigma+2}\right| \leq C \sum\limits_{\substack{j,k=1\\j\neq k}}^l |a_j||a_k|^{2\sigma+1},
\end{equation*}

the pairwise orthogonality of the family $\left\{x^j\right\}$ and Hölder's inequality, the mixed terms in right hand side vanish, then
\begin{equation*}
        m^{2\sigma+2} \leq \lim_{n \rightarrow +\infty}  \sum_{j=1}^l \left\||\cdot|^\frac{-b}{2\sigma+2} V^j(\cdot - x_n^j)\right\|_{L^{2\sigma+2}}^{2\sigma+2}.
\end{equation*}

By using the sharp Gagliardo-Nirenberg inequality (\ref{gagliardo}) and (\ref{ortho3}), we obtain
\begin{align*}
        m^{2\sigma+2} & \leq K_{opt} \sup_{j \geq 1} \left\{\|V^j\|_{L^2}^{2\sigma+2-(N\sigma+b)} \right\}\sum_{j=1}^\infty \|\nabla V^j\|_{L^2}^{N\sigma+b}\\
        & \leq K_{opt} \sup_{j \geq 1} \left\{\|V^j\|_{L^2}^{2\sigma+2-(N\sigma+b)} \right\} \sup_{k \geq 1} \left\{\|\nabla V^k\|_{L^2}^{N\sigma+b-2} \right\} \sum_{j=1}^{\infty} \|\nabla V^j\|_{L^2}^2 \\
        &\leq K_{opt} \sup_{j \geq 1} \left\{\|V^j\|_{L^2}^{2\sigma+2-(N\sigma+b)} \right\} M^{N\sigma+b-2} M^2 \\
        &= K_{opt} \sup_{j \geq 1} \left\{\|V^j\|_{L^2}^{2\sigma+2-(N\sigma+b)} \right\} M^{N\sigma+b}.
\end{align*}

Thus,
\begin{equation*}
        \sup_{j \geq 1} \left\{\|V^j\|_{L^2}^{2\sigma+2-(N\sigma+b)}\right\}\geq\frac{ m^{2\sigma+2}}{K_{opt} M^{N\sigma+b} }.
\end{equation*}

Since the series $\sum_j\|V^j\|_{L^2}^2$ converges by (\ref{ortho2}), the supremum above is attained by, say, $V^{j_0}$. Therefore,
\begin{equation*}
        \|V^{j_0}\|_{L^2}^{2\sigma+2-(N\sigma+b)} \geq \frac{m^{2\sigma+2}}{K_{opt} M^{N\sigma+b} }.
\end{equation*}

In the $L^2$-critical setting, $\sigma=\frac{2-b}{N}$, this inequality reduces to
\begin{equation*}
         \|V^{j_0}\|_{L^2}^\frac{4-2b}{N} \geq \frac{m^{\frac{4-2b}{N}}+2}{K_{opt} M^2 }. 
 \end{equation*}
Therefore, by \ref{gagliardo_opt}
\begin{equation*}
         \|V^{j_0}\|_{L^2} \geq C(M,m)\|Q\|_{L^2},
 \end{equation*}
 where $C(M,m)$ is given by \eqref{CMn}.

On the other hand, we also have
\begin{equation*}
        v_n(x+x_n^{j_0}) = V^{j_0}(x) + \sum\limits_{\substack{1 \leq j \leq l\\ j \neq j_0}} V^j(x+x_n^{j_0}-x_n^j)+ v_n^l(x+x_n^{j_0}).
\end{equation*}

The pairwise orthogonality of $\left\{x^j\right\}$ implies
\begin{equation*}
        V^j(\cdot+x_n^{j_0}-x_n^j) \rightharpoonup 0
\end{equation*}

for every $j \neq j_0$. Therefore,
\begin{equation*}
        v_n(\cdot + x_n^{j_0}) \rightharpoonup V^{j_0} + \tilde{v}^l,
\end{equation*}

where $\tilde{v}^l$ denotes the weak limit, up to a subsequence, of $\left\{v_n(\cdot + x_n^{j_0})\right\}$.

Moreover, by \eqref{ortho1}, we have
\begin{equation*}
        \|\tilde{v}^l\|_{L^{2\sigma+2}} \leq \limsup_{n \rightarrow +\infty} \|v_n^l\|_{L^{2\sigma+2}} \xrightarrow[l \rightarrow + \infty]{} 0.
\end{equation*}

Thus, by uniqueness, $\tilde{v}^l = 0$ for every $l \geq j_0$ and
\begin{equation*}
        v_n(\cdot + x_n^{j_0}) \rightharpoonup V^{j_0}.
\end{equation*}

Setting $x_n:=x_n^{j_0}$ and $V:=V^{j_0}$ we complete the proof of Theorem \ref{teo1}.

\end{proof}

%\input{introduction}
%\input{mass_concentration}
%\input{minimal_mass}
%\input{critical_norm}
%\input{appendix}

%%%%%%%%%%%%%%%%%%%%%%%%%%%%%%%%%%%%%%%%%%%%%%%%%%%%%%%%%%%%%
%% BIBLIOGRAPHY AND OTHER LISTS
%%%%%%%%%%%%%%%%%%%%%%%%%%%%%%%%%%%%%%%%%%%%%%%%%%%%%%%%%%%%%
%% A small distance to the other stuff in the table of contents (toc)
\addtocontents{toc}{\protect\vspace*{\baselineskip}}

%% The Bibliography
%% ==> You need a file 'literature.bib' for this.
%% ==> You need to run BibTeX for this (Project | Properties... | Uses BibTeX)
%\addcontentsline{toc}{chapter}{Bibliography} %'Bibliography' into toc
%\nocitep{*} %Even non-citepd BibTeX-Entries will be shown.
%\bibliographystyle{alpha} %Style of Bibliography: plain \ apalike \ amsalpha \ ...
%\bibliography{literature} %You need a file 'literature.bib' for this.

%%%%%%%%%%%%%%%%%%%%%%%%%%%%%%%%%%%%%%%%%%%%%%%%%%%%%%%%%%%%%
%% APPENDICES
%%%%%%%%%%%%%%%%%%%%%%%%%%%%%%%%%%%%%%%%%%%%%%%%%%%%%%%%%%%%%
\appendix
%% ==> Write your text here or include other files.

%\input{FileName} %You need a file 'FileName.tex' for this.
%\nocitep{*}

\bibliography{biblio}
\bibliographystyle{plainnat}
\newcommand{\Addresses}{{% additional braces for segregating \footnotesize
  \bigskip
  \footnotesize

  L. Campos, \textsc{Department of Mathematics, UFMG, Brazil}\par\nopagebreak
  \textit{E-mail address:} \texttt{me@luccascampos.com.br}

  \medskip

  M. Cardoso, \textsc{Department of Mathematics, UFMG, Brazil;}
  \textsc{Department of Mathematics, UFPI, Brazil}\par\nopagebreak
  \textit{E-mail address:} \texttt{mykael@ufpi.edu.br}

}}
\setlength{\parskip}{0pt}
\Addresses

\end{document}